\documentclass{article}

\usepackage{url}

\usepackage{amssymb}
\usepackage{amsthm}
\newtheorem*{teo}{Theorem}

\usepackage{graphicx}

\usepackage{booktabs}

\newcommand{\RP}{\mathbb{RP}}
\newcommand{\matR}{\mathbb{R}}
\newcommand{\matH}{\mathbb{H}}
\newcommand{\calC}{{\cal C}}
\newcommand{\ptwoirred}{$\mathbb{P}^2$-irreducible}
\newcommand{\Lthreeone}{L_{3,1}}
\newcommand{\Lfourone}{L_{4,1}}

\newcommand{\timtil}{\begin{picture}(12,12)
\put(2,0){$\times$}\put(2,4.5){$\sim$}\end{picture}}
\newcommand{\co}{\colon\thinspace}

\title{Orientable closed 3-manifolds with surface-complexity one}

\author{Gennaro {\sc Amendola}\footnote{Supported by a Type~A Research Fellowship of the Department of Mathematics and Applications of the University of Milano-Bicocca.}}

\begin{document}

\maketitle

\begin{abstract}
After a short summary of known results on surface-complexity of closed 3-manifolds, we will classify all closed orientable 3-manifolds with surface-complexity one.
\end{abstract}

\begin{center}
\end{center}

\begin{center}
\end{center}

\section*{Introduction}

An approach to the study of closed 3-manifolds consists in filtering them.
The aim is to find a function from the set of closed 3-manifolds to the set of natural (or real positive) numbers, so that the number associated to a closed 3-manifold is a measure of how complicated the manifold is.
For closed surfaces, this can be achieved by means of genus.
For closed 3-manifolds, the problem has been studied very much and many possible functions has been found.
For example, the Heegaard genus, the Gromov norm, the Matveev complexity have been considered.

All these functions fulfil many properties.
For instance, they are additive under connected sum.
However, some of them have drawbacks.
The Heegaard genus and the Gromov norm are not finite-to-one, while the Matveev complexity is.
Hence, in order to carry out a classification process, the third one is more suitable than the first two.
The Matveev complexity is also a natural measure of how complicated the manifold is, because if a closed 3-manifold is \ptwoirred\ and different from the sphere $S^3$, the projective space $\RP^3$ and the lens space $\Lthreeone$, then its Matveev complexity is the minimal number of tetrahedra in a triangulation of the manifold (the Matveev complexity of $S^3$, $\RP^3$ and $\Lthreeone$ is zero).
Such functions are also potential tools for inductive proofs.

The author~\cite{amendola} defined another function (called {\em surface-complexity}), from the set of closed 3-manifolds to the set of natural numbers, by means of triple points of particular immersions of closed surfaces.
In this paper, we will give a short summary of known results on surface-complexity of closed 3-manifolds and we will classify all closed orientable 3-manifolds of surface-complexity one.
Those with surface-complexity zero have been classified in~\cite{amendola}, and the irreducible ones are $S^3$, $\RP^3$ and the lens space $\Lfourone$.
Among those with surface-complexity one there are 11 irreducible ones, which are listed in Table~\ref{table:list}.
\begin{table}[t]
  \begin{center}
  \begin{tabular}{l@{\hspace{1.2cm}}l@{\hspace{1.2cm}}l}
  \toprule
  lens & other elliptic & flat \\
  \midrule
  $L_{6,1}$ & $(S^2,(2,1),(2,1),(2,-1) )$ & $T^3$\\
  $L_{8,3}$ & $(S^2,(2,1),(2,1),(3,-2) )$ & $(S^2,(2,1),(4,1),(4,-3) )$\\
  $L_{12,5}$ & & $(S^2,(3,1),(3,1),(3,-2) )$\\
  $L_{14,3}$ & & $(S^2,(2,1),(2,1),(2,1),(2,-3) )$\\
  & & $(RP^2,(2,1),(2,-1) )$\\
  \bottomrule
  \end{tabular}
  \end{center}
  \caption{Irreducible orientable closed 3-manifolds with surface-complexity one.}
  \label{table:list}
\end{table}
(For Seifert manifolds we have used the orbit invariants.)
The list up to complexity two has been obtained independently by Kazakov~\cite{Kazakov}.

Vigara~\cite{Vigara:calculus} used triple points of particular transverse immersions of connected closed surfaces to define the triple point spectrum of a 3-manifold.
The definition of the surface-complexity is similar to Vigara's one, but it has the advantage of being more flexible.
This flexibility has allowed the proof of many properties fulfilled by the surface-complexity, such as finiteness, naturalness and subadditivity.

In the Appendix we will give the proof of a theorem stated in~\cite{amendola} without a proof,
which had to be given in a subsequent paper.
For this purpose, we must distinguish the orientable case from the non-orientable one.
In the former case the result is a weaker formulation of the classification of orientable 3-manifolds with surface-complexity one.
For the sake of completeness, we will also give the proof of the non-orientable part, even if it is not strictly related to the orientable case treated in this paper.

\paragraph{Acknowledgements}
I am very grateful to Sergei Matveev and Vladimir Tarkaev for useful discussions, and to the referee for his or her useful comments and corrections.
I would also like to thank the Department of Mathematics and Applications of the University of Milano-Bicocca for the nice welcome.

\section{Definitions}

Throughout this paper, all 3-manifolds are assumed to be connected and closed.
By $M$, we will always denote such a (connected and closed) 3-manifold.
Using the {\em Hauptvermutung}, we will freely intermingle the differentiable,
piecewise linear and topological viewpoints.

\paragraph{Dehn surfaces}
A subset $\Sigma$ of $M$ is said to be a {\em Dehn surface of $M$}~\cite{Papa}
if there exists an abstract (possibly non-connected) closed surface $S$ and a transverse immersion $f\co S\to M$ such that $\Sigma = f(S)$.
By transversality, in $\Sigma$ there are only the three types of points shown in Fig.~\ref{fig:neigh_Dehn_surf} (called {\em simple}, {\em double} and {\em triple}, respectively).
\begin{figure}[t]
  \centerline{
  \begin{tabular}{ccc}
    \begin{minipage}[c]{3.5cm}{\small{\begin{center}
        \includegraphics{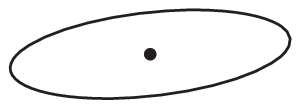}
      \end{center}}}\end{minipage} &
    \begin{minipage}[c]{3.5cm}{\small{\begin{center}
        \includegraphics{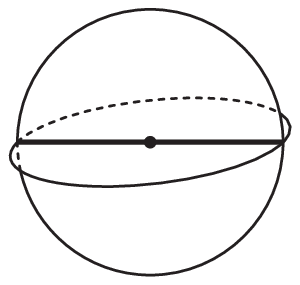}
      \end{center}}}\end{minipage} &
    \begin{minipage}[c]{3.5cm}{\small{\begin{center}
        \includegraphics{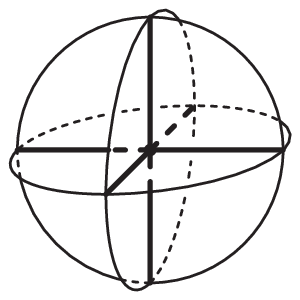}
      \end{center}}}\end{minipage} \\
    \begin{minipage}[t]{3.5cm}{\small{\begin{center}
        Simple\\point
      \end{center}}}\end{minipage} &
    \begin{minipage}[t]{3.5cm}{\small{\begin{center}
        Double\\point
      \end{center}}}\end{minipage} &
    \begin{minipage}[t]{3.5cm}{\small{\begin{center}
        Triple\\point
      \end{center}}}\end{minipage}
  \end{tabular}}
  \caption{Neighbourhoods of points (marked by thick dots) of a Dehn surface.}
  \label{fig:neigh_Dehn_surf}
\end{figure}
The set of triple points is denoted by $T(\Sigma)$; non-simple points are called {\em singular} and their set is denoted by $S(\Sigma)$.
In all figures, triple points are always marked by thick dots and the singular set is also drawn thick.

\paragraph{(Quasi-)filling Dehn surfaces and surface-complexity}
A Dehn surface $\Sigma$ of $M$ will be called {\em quasi-filling} if $M \setminus \Sigma$ is made up of balls.
Moreover, $\Sigma$ is called {\em filling}~\cite{Montesinos} if its singularities induce a cell-decomposition of $M$; more precisely,
\begin{itemize}
\item $T(\Sigma) \neq \emptyset$,
\item $S(\Sigma) \setminus T(\Sigma)$ is made up of intervals (called {\em
    edges}),
\item $\Sigma \setminus S(\Sigma)$ is made up of discs (called {\em regions}),
\item $M \setminus \Sigma$ is made up of balls (i.e.~$\Sigma$ is quasi-filling).
\end{itemize}

Since $M$ is connected, the quasi-filling Dehn surface $\Sigma$ is connected.
Moreover, $M$ minus some (suitably chosen) balls is a regular neighbourhood of $\Sigma$ and hence collapses to $\Sigma$.
It is by now well-known that a filling Dehn surface determines $M$ up to homeomorphism and that every $M$ has filling Dehn surfaces (see, for instance, Montesinos-Amilibia~\cite{Montesinos} and Vigara~\cite{Vigara:present}, see also~\cite{Amendola:surf_inv}).

\paragraph{Surface-complexity}
The {\em surface-complexity} $sc(M)$ of $M$ is equal to $c$ if $M$ possesses a quasi-filling Dehn surface with $c$ triple points and has no quasi-filling Dehn surface with less than $c$ triple points~\cite{amendola}.

\section{Properties}
\label{sec:properties}

In this section we will describe some results on surface-complexity and minimal quasi-filling Dehn surfaces.
Details and proofs, unless explicitly stated, can be found in~\cite{amendola}.

\paragraph{Minimality and finiteness}
A quasi-filling Dehn surface $\Sigma$ of $M$ is called {\em minimal} if it has a minimal number of triple points among all quasi-filling Dehn surfaces of $M$, i.e.~$|T(\Sigma)|=sc(M)$.

Let us give some examples of quasi-filling Dehn surfaces without triple points, which are clearly minimal.
Only two topological surfaces are quasi-filling Dehn surfaces of a 3-manifold: the sphere $S^2$ and the projective plane $\RP^2$, which are quasi-filling Dehn surfaces of the sphere $S^3$ and the projective space $\RP^3$, respectively.
Two projective planes intersecting along a loop non-trivial in both of them form a quasi-filling Dehn surface of $\RP^3$, which will be called {\em double projective plane}.
A sphere intersecting a torus (resp.~a Klein bottle) along a loop is a quasi-filling Dehn surface of $S^2\times S^1$ (resp.~$S^2 \timtil S^1$).
The {\em quadruple hat} (i.e.~a disc whose boundary is glued four times along a circle) is a quasi-filling Dehn surface of the lens-space $\Lfourone$.
Therefore, we have that $S^3$, $\RP^3$, $S^2\times S^1$, $S^2 \timtil S^1$ and $\Lfourone$ have surface-complexity zero.

Minimal quasi-filling Dehn surfaces without triple points are clearly not filling, but if we take into account only \ptwoirred\ 3-manifolds (except for three ones), we have minimal filling Dehn surfaces.
More precisely, suppose $M$ is \ptwoirred, then we have two cases:
\begin{itemize}
\item
if $sc(M)=0$, then $M$ is $S^3$, $\RP^3$ or $\Lfourone$;
\item
if $sc(M)>0$, then $M$ has a minimal filling Dehn surface.
\end{itemize}
Note that the only 3-manifolds that are prime but not irreducible are $S^2\times S^1$ and $S^2 \timtil S^1$, hence the theorem implies that every prime 3-manifold, except the five manifolds described above, has a minimal filling Dehn surface.

Since there is a finite number of filling Dehn surfaces having a fixed number of triple points, we have that for any integer $c$ there exists only a finite number of \ptwoirred\ 3-manifolds having surface-complexity $c$.

\paragraph{Minimal quasi-filling Dehn surfaces}
Not all the minimal quasi-filling Dehn surfaces of a \ptwoirred\ 3-manifold are indeed filling.
However, they can be all constructed starting from filling ones (except for $S^3$, $\RP^3$ and $\Lfourone$, for which non-filling ones must be used) and applying a simple move.
The move acts on a quasi-filling Dehn surface near a simple point as shown in Fig.~\ref{fig:bubble_move} and is called a {\em bubble-move}.
\begin{figure}[t]
  \centerline{\includegraphics{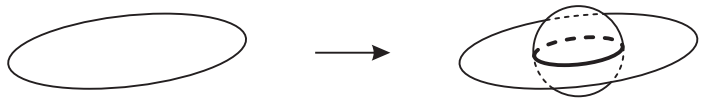}}
  \caption{Bubble-move.}
  \label{fig:bubble_move}
\end{figure}
If a quasi-filling Dehn surface $\Sigma$ is obtained from a quasi-filling Dehn surface $\overline{\Sigma}$ by repeatedly applying bubble-moves, we will say that $\Sigma$ {\em is derived from} $\overline{\Sigma}$.

The result above on minimal filling Dehn surfaces in the \ptwoirred\ case can be improved by means of a slightly subtler analysis.
Suppose $M$ is \ptwoirred.
If $\Sigma$ is a minimal quasi-filling Dehn surface of $M$, we have the following cases:
\begin{itemize}
\item
If $sc(M)=0$, one of the following holds:
\begin{itemize}
\item
$M$ is $S^3$ and $\Sigma$ is derived from $S^2$,
\item
$M$ is $\RP^3$ and $\Sigma$ is derived from $\RP^2$ or from the double projective plane,
\item
$M$ is $\Lfourone$ and $\Sigma$ is derived from the four-hat.
\end{itemize}
\item
If $sc(M)>0$, then $\Sigma$ is derived from a minimal filling Dehn surface of $M$.
\end{itemize}

\paragraph{Cubulations and naturalness}
A {\em cubulation} of $M$ is a cell-decomposition of $M$ such that
\begin{itemize}
\item each 2-cell (called a {\em face}) is glued along 4 edges,
\item each 3-cell (called a {\em cube}) is glued along 6 faces arranged like the
  boundary of a cube.
\end{itemize}
Note that self-adjacencies and multiple adjacencies are allowed.

The following construction is well-known (see~\cite{Aitchison-Matsumotoi-Rubinstein, Funar, Babson-Chan}, for instance).
A filling Dehn surface $\Sigma$ of $M$ can be constructed from a cubulation $\calC$ of $M$ by considering for each cube of $\calC$ the three squares shown in Fig.~\ref{fig:cube_to_surf} and 
by gluing them together (up to isotopy, we can suppose that the squares fit together through the faces).
\begin{figure}[t]
  \centerline{\includegraphics{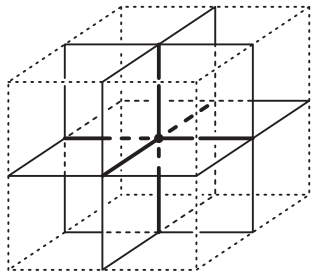}}
  \caption{Local behaviour of duality.}
  \label{fig:cube_to_surf}
\end{figure}
Conversely, a cubulation $\calC$ of $M$ can be constructed from a filling
Dehn surface $\Sigma$ of $M$ by considering an abstract cube for each triple
point of $\Sigma$ and by gluing the cubes together along the faces (the identification of each pair of faces is chosen by following the four germs of regions adjacent
to the respective edge of $\Sigma$).
The cubulation and the filling Dehn surface constructed in such a way are said
to be {\em dual} to each other.

The construction above allowed to prove that if $M$ is \ptwoirred\ and is different from $S^3$, $\RP^3$ and $\Lfourone$, then its surface-complexity is equal to the minimal number of cubes in a cubulation of $M$.

\paragraph{Subadditivity}
An important feature of a complexity function is to behave well with respect to the
cut-and-paste operations.
The surface-complexity is subadditive under connected sum.
Namely, the surface-complexity of the connected sum of 3-manifolds is less than or equal to the sum of their surface-complexities.
We do not know whether it is indeed additive.

\paragraph{Estimations}
In general, calculating the surface-complexity $sc(M)$ of $M$ is very difficult, but estimating it is relatively easy.
More precisely, it is quite easy to give upper bounds for it.
If one constructs a quasi-filling Dehn surface $\Sigma$ of $M$, the number of triple points of $\Sigma$ is an upper bound for the surface-complexity of $M$.
Afterwards, the (usually difficult) problem of proving the sharpness of this bound arises.

There are explicit constructions of quasi-filling Dehn surfaces of $M$ starting from triangulations, Heegaard splittings and Dehn surgery presentations of $M$.
They allow to prove some estimations.

\begin{description}
\item{\em Triangulations}
Suppose that $M$ has a triangulation with $n$ tetrahedra.
Then, the inequality $sc(M) \leqslant 4n$ holds.

\item{\em Heegaard splittings}
Suppose that $H_1 \cup H_2$ is a Heegaard splitting of $M$ such that the meridians of the handlebody $H_1$ intersect those of $H_2$ transversely in $n$ points.
Then, the inequality $sc(M) \leqslant 4n$ holds.
This estimation can be improved for any connected sum of \ptwoirred\ 3-manifolds $M_k$, such that no $M_k$ is $\Lthreeone$.
Indeed, suppose that $H_1 \cup H_2$ is a Heegaard splitting of such a 3-manifold $M$, that the meridians of the handlebody $H_1$ intersect those of $H_2$ transversely in $n$ points, and that the closure of one of the components into which the meridians of $H_1$ and $H_2$ divide $\partial H_1 = \partial H_2$ contains $m$ of these points.
Then, the inequality $sc(M) \leqslant 4n-4m$ holds.

\item{\em Dehn surgery}
Suppose that $M$ is obtained by Dehn surgery along a framed link $L$ in $S^3$ (hence $M$ is orientable).
Moreover, suppose that $L$ has a projection such that the framing is the blackboard one, such that there are $n$ crossing points, and such that there are $m$ components containing no overpass.
Then, the inequality $sc(M) \leqslant 8n + 4m$ holds.
If the framing is not the blackboard one, we have
$sc(M) \leqslant 8n + 4m + 4\sum_i |fr_i-w_i|$,
where $fr_i$ and $w_i$ are, respectively, the framing and the wirthe of the $i$-th component of the link.
\end{description}

\paragraph{Matveev complexity}
The surface-complexity is related to the Matveev complexity, at least in the \ptwoirred\ case.
The latter is defined using simple spines.
A polyhedron $P$ is {\em simple} if the link of each point of $P$ can be embedded in the 1-skeleton of the tetrahedron.
The points of $P$ whose link is the whole 1-skeleton of the tetrahedron are called {\em vertices}.
A sub-polyhedron $P$ of $M$ is a {\em spine} of $M$ if $M \setminus P$ is a ball.
The {\em Matveev complexity} $c(M)$ of $M$ is the minimal number of vertices of a simple spine of $M$.
The interested reader is referred to Matveev~\cite{Matveev:book} for a complete discussion on Matveev complexity.

Suppose that $M$ is \ptwoirred\ and different from $\Lthreeone$ and $\Lfourone$.
Then, the inequalities
$sc(M) \leqslant 4c(M)$
and
$c(M) \leqslant 8sc(M)$
hold.
The latter inequality has been improved in the orientable case by Tarkaev~\cite{Tarkaev}, who proved that $c(M) \leqslant 6sc(M)$ holds.
For the two missing manifolds, we have $c(\Lthreeone)=0$, $sc(\Lthreeone)=2$, $c(\Lfourone)=1$ and $sc(\Lfourone)=0$.

\section{Orientable 3-manifolds with surface-complexity one}
\label{sec:sc_one}

In this section we will describe how we have obtained the list of the orientable 3-manifolds with surface-complexity one.
The list of the irreducible orientable 3-manifolds with surface-complexity one is shown in Table~\ref{table:list}.
Note that six are elliptic and five are flat.
Note also that the only missing flat 3-manifold is $(S^2,(2,1),(3,1),(6,-5) )$.
Finally, note that there are no \ptwoirred\ orientable 3-manifolds having surface-complexity one of geometric-type $\matH^2\times\matR$, $\widetilde{{\rm SL}_2\matR}$, Sol, hyperbolic or non-geometric.

We will now explain the steps to obtain this list; we will not go into detail.
By means of duality, we can list cubulations with one cube up to homeomorphism.
There are three inequivalent ways to pair the six faces of a cube.
Each way involves three gluings of faces and each gluing can be done in four ways.
Therefore, we need to analyse $3\times 4^3=192$ cubulations.
Among these we have ruled out those whose underlying topological space is not a closed 3-manifold, and we have removed duplicates (up to isomorphism of the cubic structure); in order to carry out this step, we have used a simple computer program.
We have found 29 non-isomorphic cubulations of orientable 3-manifolds.

We have then examined each of these 29 cubulations to identify the underlying manifolds.
Giving a name to many cubulations has been quite easy.
Sometimes the summands of a connected sum and the Seifert structure are clear from the cubic structure.
For the other cases we have recognised the manifold by computing some topological invariants (homology, homotopy, $\epsilon$-invariant~\cite{Matveev-Ovchinnikov-Sokolov,Matveev:book}, Matveev complexity) and by searching in the list of Matveev~\cite{Matveev:book}.

In order to obtain the list of all orientable 3-manifolds with surface-complexity one, we apply the results on minimality and subadditivity described in the previous section.
Since the only 3-manifolds that are prime but not irreducible are $S^2\times S^1$ and $S^2 \timtil S^1$, each orientable 3-manifold with surface-complexity one is the connected sum of one irreducible orientable 3-manifold with surface-complexity one (listed in Table~\ref{table:list}) and some (possibly none) among $\RP^3$, $\Lfourone$, $S^2\times S^1$ and $S^2 \timtil S^1$.

\appendix

\section{Partial results in the non-orientable case}

We give here the proof of Theorem~A.1 of~\cite{amendola}.
\begin{teo}\hspace*{1cm}
\begin{itemize}
\item
There are no \ptwoirred\ orientable 3-manifolds having surface-complexity one of geometric-type $\matH^2\times\matR$, $\widetilde{{\rm SL}_2\matR}$, Sol, hyperbolic or non-geometric.
\item
There are \ptwoirred\ elliptic orientable 3-manifolds having surface-com\-plexity one (e.g.~the lens spaces $L_{6,1}$, $L_{8,3}$, $L_{14,3}$, $L_{12,5}$) and flat ones (e.g.~the 3-dimensional torus $T^3$).
\item
There are no \ptwoirred\ non-orientable 3-manifolds having surface-com\-plexity one of geometric-type $\matH^2\times\matR$, hyperbolic or non-geometric.
\item
There are \ptwoirred\ flat non-orientable 3-manifolds having surface-com\-plexity one (e.g.~the trivial bundle over the Klein bottle $K \times S^1$).
\end{itemize}
\end{teo}

\begin{proof}
The first two points follow from the classification of orientable 3-manifolds with surface-complexity one, described in the Section~\ref{sec:sc_one}.

In order to prove the third point, we apply Lemma~2 of~\cite{Tarkaev}, which holds also in the non-orientable case and can be stated as follows:\\
{\em If a 3-manifold has a cubulation with one cube, its Mateveev complexity is at most six.}

Since no \ptwoirred\ 3-manifold of geometric-type $\matH^2\times\matR$, hyperbolic or non-geometric appears in the census of closed non-orientable \ptwoirred\ 3-manifolds with Mateveev complexity up to six (see~\cite{Amendola-Martelli}), the third point is proved.

The fourth point follows from the fact that $K \times S^1$ has a cubulation with one cube.
\end{proof}

\noindent
\parbox{\textwidth}{
\textit{Department of Mathematics and Applications\\
University of Milano-Bicocca\\
Via Cozzi, 53, I-20125, Milano, Italy}\\[5pt]
\url{gennaro.amendola@unimib.it}\\[5pt]
\url{http://www.dm.unipi.it/~amendola/}
}

\end{document}